\documentclass[12pt]{amsart}

\usepackage[mathscr]{eucal}
\usepackage{amscd}
\usepackage{amsfonts}
\usepackage{amsmath, amsthm, amssymb}
\usepackage{latexsym}
\usepackage[dvips]{graphics}

\theoremstyle{plain}
\newtheorem{thm}{Theorem}

\newtheorem{lem}[thm]{Lemma}
\newtheorem{prop}[thm]{Proposition}

\theoremstyle{definition}
\newtheorem{defn}[thm]{Definition}

\theoremstyle{remark}
\newtheorem{rem}[thm]{Remark}

\numberwithin{equation}{section}

\title{REMARKS ON A CATEGORICAL DEFINITION \\ OF DEGENERATION IN TRIANGULATED CATEGORIES}

\author{Alexander Zimmermann} 

\address{
\begin{flushleft}
        \hspace{0.3cm}  D\'epartement de Math\'ematiques et LAMFA (UMR 7352 du CNRS)\\
         \hspace{0.3cm}  Universit\'e de Picardie\\
         \hspace{0.3cm}  33, rue St Leu \\
         \hspace{0.3cm}  80039 Amiens, FRANCE\\
\end{flushleft}
}
\email{Alexander.Zimmermann@u-picardie.fr}

 \thanks{The detailed version of this paper will be submitted
 for publication elsewhere.}

\begin{document}

\maketitle


\begin{abstract}
This work reports on joint research with Manuel Saorin.
For an algebra $A$ over an algebraically closed field $k$ the set of $A$-module structures on $k^d$ forms an affine algebraic variety. The general linear group $Gl_d(k)$ acts on this variety and isomorphism classes correspond to orbits under this action. A module $M$ degenerates to a module $N$ if $N$ belongs to the Zariski closure of the orbit of $M$. Yoshino gave a scheme-theoretic characterisation, and Saorin and Zimmermann generalise this concept to general triangulated categories. We show that this concept has an interpretation in terms of distinguished triangles, analogous to the Riedtmann-Zwara characterisation for modules. In this manuscript we report on these results and study the behaviour of this degeneration concept under functors between triangulated categories.




\end{abstract}

\section{Introduction}

Already very early in representation theory of algebras a geometric
interpretation of representations of an algebra was given, cf e.g. work of
Gabriel~\cite{Gabriel}. For an algebraically closed field $k$, a
finite-dimensional $k$-algebra $A$ and some integer $d>0$ the set of $A$-module structures on $k^d$ forms an affine algebraic variety $mod(A,d)$. The general linear group $Gl_d(k)$ acts on this variety and two $A$-module structures on $k^d$ are isomorphic if and only if they belong to the same orbit. One says that an
$A$-module $M$ degenerates to the module $N$ if $N$ belongs to the Zariski closure of the orbit of $M$. Much work was done to explain the geometric structure of the orbit closures. Riedtmann \cite{Riedtmann} and Zwara \cite{Zwara} prove that $M$
degenerates to $N$ if and only if there is an $A$-module $Z$ and an embedding of
$Z$ into $M\oplus Z$ such that $N$ is isomorphic to the cokernel of this
embedding. We refer to Section~\ref{Sectionclassical} for more details on this
part of the theory.

Yoshino studies in \cite{YoshinoCM, YoshinoModules, Yoshinostable}
degeneration for more general algebras, including maximal Cohen-Macaulay
modules over a local Gorenstein $k$-algebra, and for this purpose he gave a
scheme-theoretic definition of this concept. Using this concept
Yoshino studies stable categories
of maximal Cohen-Macaulay modules over a local Gorenstein algebra.
We refer to Section~\ref{Yoshinosconceptsect} for more details.

Yoshino's concept is then suitable for general triangulated categories.
In joint work with Saorin \cite{SaorinZim} we define a degeneration
concept for general triangulated $k$-categories with splitting idempotents.
We then show that this concept implies in a very general setting that if an
object $M$ degenerates to an object $N$, then there is an object $Z$ and a
distinguished triangle
$Z\stackrel{v\choose u}{\rightarrow} Z\oplus M\rightarrow N\rightarrow Z[1]$
in the triangulated category with nilpotent endomorphism $v$ of $Z$.
We then write $M\leq_\Delta N$.
For algebraic triangulated categories, and an additional technical
hypothesis, we prove the converse. We study in the present paper the
question what happens if two objects $M$ and $N$ belonging to a
triangulated category ${\mathcal T}_1$
such that ${\mathcal T}_1$ is a full triangulated
subcategory of a triangulated category ${\mathcal T}_2$.
Under the hypotheses on ${\mathcal T}_2$ which
we need for the converse of the main theorem of \cite{SaorinZim}
as mentioned above, we show that then the degeneration concepts coincide.

We finally mention that our degeneration concept applies to the bounded derived
category of a finite dimensional algebra, and to the stable category of a
selfinjective algebra. For the bounded derived category over an algebra
Jensen, Su and Zimmermann gave an alternative definition in \cite{JSZdegen}.
Also in \cite{JSZdegen} we showed that degeneration there is equivalent
to the existence of a distinguished triangle as above. However, this
concept is very closely linked to the specific situation of derived
categories of bounded complexes over a finite dimensional algebra.
Moreover, no clear relation to degeneration of the homology can be seen.
In a subsequent approach Jensen, Madsen and Su \cite{JMS} used
$A_\infty$ algebras to define a degeneration by means of the
homology of a complex. Again, this is not done for general triangulated
categories.

In the classical theory degeneration of modules provides a partial
order on the isomorphism classes of objects. In \cite{JSZpartord}
Jensen, Su and Zimmermann study when the degeneration given by the
existence of a distinguished triangle
$Z\stackrel{v\choose u}{\rightarrow} Z\oplus M\rightarrow N\rightarrow Z[1]$
gives a partial order. This happens to be the case when some finiteness
conditions are assumed, in particular morphism spaces in the triangulated
category should be $k$-modules of finite length for all objects in the
triangulated category. Moreover, for two objects $X,Y$ we ask that we may find a
shift $n_{X,Y}$ such that there is no non-zero morphism from $X$ to $Y[n_{X,Y}]$.

\medskip

{\bf Acknowledgement :} I wish to thank the organisers of the
Symposion on Ring Theory and Representation theory 47,
and in particular Hideto Asashiba for the kind invitation
to Osaka, for giving me the opportunity to present my work and for the great hospitality during my visit.

\section{Classical degeneration concepts}
\label{Sectionclassical}

Degeneration between modules over a fixed algebra is a relatively classical subject in representation theory of finite dimensional algebras, and was used in many different ways.

Let $k$ be an algebraically closed field and let $A$ be a finite dimensional $k$-algebra. Then an $A$-module of $k$-dimension $d$ is an algebra homomorphism
$A\stackrel{\varphi}{\rightarrow} End_k(k^d)$. Hence, if $a_1,\cdots,a_m$ are
algebra generators of $A$, then for each $i\in\{1,\cdots,m\}$ each of the
$\varphi(a_i)=:M_i$ is a square matrix of size $d$. Moreover, $A$ is finitely
presented, in the way that there is a finite set
$\rho_1(X_1,\cdots,X_m),\cdots,\rho_s(X_1,\cdots,X_m)$ of relations such that if
$k\langle X_1,\dots,X_m\rangle$ denotes the free algebra in $m$ variables
$X_1,\cdots,X_m$, then as an algebra we get $$A\simeq k\langle
X_1,\cdots,X_m\rangle/(\rho_1,\cdots,\rho_s).$$
The points of the affine algebraic variety $mod(A,d)$ defined by the $m\cdot d^2$ variables given by the coefficients of the matrices $M_1,\cdots,M_m$ modulo the
relations given by the polynomial equations $\rho_1,\cdots,\rho_s$ parameterise
$A$-module structures on $k^d$. Two modules $N_1$ and $N_2$ corresponding to the
points $n_1,n_2$ of $mod(A,d)$ are isomorphic if and only if the corresponding
matrices $M_1(n_1),\cdots,M_m(n_1)$ are simultaneously conjugate with the matrices $M_1(n_2),\cdots,M_m(n_2)$. Otherwise said, $G:=GL_d(k)$ acts on $mod(A,d)$ by
matrix conjugation and the points $n_1$ and $n_2$ correspond to isomorphic modules if and only if $n_1$ and $n_2$ belong to the same orbit $G\cdot n_1=G\cdot n_2$
under this action. In general orbits are not Zariski closed, and we denote by
$\overline{G\cdot n}$ the Zariski closure of the orbit $G\cdot n$. Now, we say
that {\em $N_1$ degenerates to $N_2$} if and only if $n_2\in \overline{G\cdot
n_1}$. In this case we note $N_1\leq_{deg}N_2$.

An algebraic classification of degenerations was subject of intensive research.
It is relatively easy to see that $N_2\simeq N/N_1$ implies $N\leq_{deg}N_1\oplus N_2$. Moreover, $N_1\leq_{deg}N_2$ implies $N_3\oplus N_1\leq_{deg}N_3\oplus N_2$
for all $A$-modules $N_1,N_2,N_3$. The converse is not true, as may be shown by an example due to Jon Carlson (cf \cite[\S{} 3.1]{Riedtmann}); another example was given with different methods by Yoshino~\cite[Proposition 3.3]{Yoshinostable}.
Further, if $N_1\leq_{deg}N_2$, then $$\dim_k(Hom_A(X,N_1))\leq \dim_k(Hom_A(X,N_2))$$
and $$\dim_k(Hom_A(N_1,X))\leq \dim_k(Hom_A(N_2,X))$$
for all $X$. This property implies actually that $\leq_{deg}$ is a partial order, as was shown by Auslander~\cite{Auslander}. An independent proof was later given by Bongartz~\cite{Bongartz}, and an adaption of this proof was used in \cite{JSZpartord} to show that $\leq_\Delta$ is a partial order under some reasonable hypotheses.
Riedtmann showed in \cite{Riedtmann} that if there is an $A$-module $Z$ and a short exact sequence
$$0\rightarrow Z\rightarrow Z\oplus N_1\rightarrow N_2\rightarrow 0$$ then $N_1\leq_{deg}N_2$, and Zwara showed in \cite{Zwara} the converse in this generality. The above relations on the dimension of $Hom$-spaces is an easy consequence, though it was proved earlier by different methods.

\section{On Yoshino's degeneration concept}
\label{Yoshinosconceptsect}

The fact that we only deal with finite dimensional algebras in Section~\ref{Sectionclassical} is in some sense unsatisfying.
In order to be able to cover a greater generality, Yoshino changed
the classical degeneration $\leq_{deg}$ to a scheme theoretic concept
which is well-suited for us. We explain Yoshino's results here.

Let $k$ be a field and let $A$ be a $k$-algebra. Yoshino developed in a
series of papers a degeneration concept which is well-suited for the
purpose of commutative algebra. By the symbol $(V,tV,k)$ we denote a
discrete valuation ring $V$ with
radical $tV$ and residue field $k$.
An algebra which is a discrete valuation ring is a
discrete valuation $k$-algebra.

\begin{defn} (Yoshino \cite{YoshinoModules})
Let $A$ be a $k$-algebra and let $M$ and $N$ be two finitely generated $A$-modules.
We say {\em $M$ degenerates to $N$ along a discrete valuation ring}, and we write in this case $M\leq_{dvr}N$, if there
is a discrete valuation $k$-algebra $(V,tV,k)$ and an
$A\otimes_kV$-module $Q$, which is
\begin{itemize}
\item flat as $V$-module,
\item such that $M\otimes_kV[\frac{1}{t}]\simeq Q\otimes_VV[\frac{1}{t}]$ as $A\otimes_kV[\frac{1}{t}]$-modules and
\item such that $N\simeq Q/tQ$ as $A$-modules.
\end{itemize}
\end{defn}

The interpretation of this notion is that there is an affine line,
presented by $V$, and a point $Q$ that moves along $V$. The algebra
$V$ is a discrete valuation algebra since we are only interested in
the neighbourhood of the parameter $t=0$. Now, at the value $t=0$ the
moving point $Q$ becomes $Q/tQ$, which is assumed to be isomorphic to
$N$, and generically, outside $t=0$, the moving point looks like $M$.
This last fact is expressed by the condition
$M\otimes_kV[\frac{1}{t}]\simeq Q\otimes_VV[\frac{1}{t}]$.

Of course, Yoshino's concept $M\leq_{dvr} N$ immediately
generalises to
the stable category. The only thing to do is to replace the
isomorphisms
in the module category by isomorphisms in the stable categories.
Yoshino formulated this concept for stable categories of maximal
Cohen-Macaulay modules over local Gorenstein rings.
A local commutative ring $A$ with residue field $k$ is a Gorenstein
ring if $A$ is Noetherian with finite injective dimension.
In this case an $A$-module $M$ is Cohen-Macaulay if $Ext^i_A(M,A)=0$
for all $i>0$.
It is well-known that the stable category of maximal Cohen-Macaulay
modules over a local Gorenstein $k$-algebra is triangulated.

\begin{defn} (Yoshino) \cite{Yoshinostable}\label{yoshinogeostabledefinition}
Let $k$ be a field, and let $(A,{\mathfrak m},k)$ be a local
Gorenstein $k$-algebra and let $M$ and $N$ be two $A$-modules.
We say {\em $M$ stably degenerates to $N$ along a discrete valuation ring }
if there is a discrete valuation $k$-algebra $(V,tV,k)$ and a maximal
Cohen-Macaulay $A\otimes_kV$-module $Q$, such that
\begin{itemize}
\item  $M\otimes_kV[\frac{1}{t}]\simeq Q\otimes_VV[\frac{1}{t}]$ in the
stable category of maximal Cohen-Macaulay $A\otimes_kV[\frac{1}{t}]$
modules and
\item  $N\simeq Q/tQ$ in the stable category of maximal
Cohen-Macaulay $A$-modules.
\end{itemize}
In this case we write $M\leq_{stdvr}N$.
\end{defn}

Now, the most striking fact is that this concept implies, and is in some
cases actually equivalent to an analogue of Riedtmann-Zwara's
characterisation in terms of short exact sequences.

\begin{defn}
\begin{itemize}
\item
Let $\mathcal A$ be an abelian category. We say that an object
$M$ degenerates to an object $N$ if there is an object $Z$ and a short exact sequence $$0\rightarrow Z\stackrel{v\choose u}{\rightarrow}Z\oplus M\rightarrow N\rightarrow 0$$
with a nilpotent endomorphism $v$ of $Z$. We write $M\leq_{RZ}N$ in this case.
\item
Let $\mathcal T$ be a triangulated category with suspension functor denoted by $T$. We say that an object
$M$ degenerates to an object $N$ if there is an object $Z$ and a distinguished triangle $$ Z\stackrel{v\choose u}{\rightarrow}Z\oplus M\rightarrow N\rightarrow Z[1]$$
with a nilpotent endomorphism $v$ of $Z$. We write $M\leq_\Delta N$ in this case
\end{itemize}
\end{defn}

Riedtmann and Zwara considered this degeneration for modules $M$ and $N$ over a finite dimensional algebras $A$ over a field $k$. In this case Fitting's lemma implies that the hypothesis on $v$ to be nilpotent is not necessary, and actually these authors do not assume that $v$ is nilpotent. Up to my knowledge the importance of this nilpotence hypothesis was first observed by Yoshino \cite{YoshinoModules}.

Yoshino gave as a main theorem of \cite{YoshinoModules,Yoshinostable} the following result. Recall that the stable category of maximal Cohen-Macaulay modules over a local Gorenstein $k$-algebra is triangulated. In particular $\leq_{stdvr}$ and $\leq_\Delta$ are both defined for this category.

\begin{thm}\label{yoshinostheorem}
\begin{itemize}\item\cite{YoshinoModules}
Let $k$ be a field and let $A$ be a $k$-algebra. Let $M$ and $N$ be finitely generated $A$-modules. Then $$\left(M\leq_{dvr}N\right)\Leftrightarrow \left(M\leq_{RZ}N\right).$$
\item\cite{Yoshinostable}
Let $k$ be a field, and let $(R,{\mathfrak m},k)$ be a local Gorenstein $k$-algebra.
Let $\mathcal T$ be the stable category of maximal Cohen Macaulay $R$-modules and let $M$ and $N$ be two objects of $\mathcal T$. Then $$\left(\exists_{m,n\in{\mathbb N}}R^m\oplus M\leq_{dvr}R^n\oplus N\right)\Rightarrow \left(M\leq_{\Delta}N\right)\Rightarrow \left(M\leq_{stdvr} N\right).$$
Moreover, these three conditions are equivalent if $A$ is artinian.
\end{itemize}
\end{thm}

The implications may be strict in general. Yoshino gave explicit examples for the first implication.

\section{The categorical degeneration}

We shall now give a generalisation of Yoshino's degeneration concept Definition~\ref{yoshinogeostabledefinition} for the stable category of maximal Cohen-Macaulay modules.

\begin{defn}\label{degendatadef}
Let $k$ be a commutative ring and let ${\mathcal C}_k^\circ$ be a
$k$-linear triangulated category with split idempotents.
A {\em degeneration data} for ${\mathcal C}_k^\circ$ is given by
\begin{itemize}
\item a triangulated category ${\mathcal C}_k$ with split
idempotents and a fully faithful
embedding ${\mathcal C}_k^\circ\rightarrow{\mathcal C}_k$,
\item
a triangulated category   ${\mathcal C}_V$ with split idempotents
and a full triangulated subcategory ${\mathcal C}_V^\circ$,
\item  triangulated functors
$\uparrow_k^V:{\mathcal C}_k\rightarrow {\mathcal C}_V$
and  $\Phi:{\mathcal C}_V^\circ\rightarrow {\mathcal C}_k$,
such that
$({\mathcal C}_k^\circ)\uparrow_k^V\subseteq {\mathcal C}_V^\circ$,
when
we view ${\mathcal C}_k^\circ$ as a full subcategory of ${\mathcal
C}_k$,
\item  a natural transformation
$\text{id}_{{\mathcal C}_V}\stackrel{t}{\rightarrow} \text{id}_{{\mathcal C}_V}$
of triangulated functors.
\end{itemize}
These triangulated categories and functors should satisfy the following axioms:
\begin{enumerate}
\item \label{4}
For each object $M$ of ${\mathcal C}_k^\circ$ the morphism
$\Phi(M\uparrow_k^V)\stackrel{\Phi(t_{M\uparrow_k^V})}{\rightarrow}\Phi(M\uparrow_k^V)$
is a split monomorphism in ${\mathcal C}_k$.
\item \label{6}
For all objects $M$ of ${\mathcal C}_k^\circ$ we get
$\Phi(\text{cone}(t_{M\uparrow_k^V}))\simeq M$.
\end{enumerate}
\end{defn}

All throughout the paper, whenever we have a degeneration data for
$\mathcal{C}_k^\circ$ as above, we will see $\mathcal{C}_k^\circ$ as
a full subcategory of $\mathcal{C}_k$.

\begin{defn}\label{degendef}
Given two objects $M$ and $N$ of ${\mathcal C}_k^\circ$ we say that
{\em $M$ degenerates to $N$ in the categorical sense} if there is a
degeneration data for ${\mathcal C}_k^\circ$ and an object $Q$ of
${\mathcal C}_V^\circ$ such that
$$p(Q)\simeq p(M\uparrow_k^V)\mbox{ in ${\mathcal C}_V^\circ[t^{-1}]$
and }\Phi(\text{cone}(t_Q))\simeq N,$$ where $p:{\mathcal
C}_V^\circ\rightarrow{\mathcal C}_V^\circ [t^{-1}]$ is the
canonical functor. In this case we write $M\leq_{cdeg}N$.
\end{defn}

\begin{rem}
The functor $\uparrow_k^V$ models $V\otimes_k-$ from Yoshino's attempt.
The functor
$\Phi$ models the forgetful functor which is the identity on objects,
i.e. an $A\otimes_kV$-module $M$ is considered as an $A$-module only.
Of course, in the classical situation considered by Yoshino
$M$ is not finitely generated anymore. This is the reason why we
need to consider the categories ${\mathcal C}_k^\circ$ inside
${\mathcal C}_k$, and ${\mathcal C}_V^\circ$ inside ${\mathcal C}_V$.
\end{rem}

The concept $\leq_{cdvr}$ is appropriate, as we shall see in the following result.
It gives under certain conditions an equivalence between the geometric, or scheme-theoretical, degeneration $\leq_{cdeg}$ and the algebraic notion
of degeneration $\leq_\Delta$.

\begin{thm} (Saorin and Zimmermann \cite{SaorinZim})
Let $k$ be a commutative ring.
\label{main}
\begin{enumerate}
\item Let ${\mathcal C}_k^\circ$ be a triangulated $k$-category with split
idempotents and let $M$ and $N$ be two objects of ${\mathcal
C}_k^\circ$. Then $M\leq_{cdeg}N$ implies that $M\leq_\Delta N$.
\item
Suppose that $k$ is a field.
Let $\mathcal{C}_k^0$ be the category of compact objects of an
algebraic compactly generated triangulated $k$-category. If $M\leq_\Delta N$,
then $M\leq_{cdeg}N$.
\end{enumerate}
\end{thm}

We observe that this is indeed a generalisation of Yoshino's Theorem~\ref{yoshinostheorem}. Moreover, both parts of the theorem
are valid for the bounded derived category of $A$-modules for $A$ being a finite
dimensional $k$-algebra.

The reason why we need the additional hypotheses for the second part of the theorem is to be able to apply a result due to Keller~\cite{Kellerddc,Kelleralgebraic}. This result implies that then ${\mathcal C}^\circ_k$ is the subcategory of compact objects of the derived category of a differential graded $k$-category. For the first part a main difficulty is first to construct the object $Z$ of ${\mathcal C}_k$. This is done by tricky applications of octahedral axioms. Another main difficulty is to show that the object $Z$ is actually in ${\mathcal Z}^\circ_k$, and not only in ${\mathcal C}_k$. This is shown using a result due to May~(cf e.g.\cite[Lemma 3.4.5]{reptheobuch} or \cite{May}).

\section{Categorial degeneration and triangle functors}

We see immediately that if the triangulated category $\mathcal C$ is equivalent to the triangulated category $\mathcal D$, given by some functor $F$, then $$\left[M\leq_{cdeg}N\Leftrightarrow F(M)\leq_{cdeg}F(N)\right]\mbox{ and }\left[M\leq_{\Delta}N\Leftrightarrow F(M)\leq_{\Delta}F(N)\right].$$
However, if $F$ is not an equivalence the situation is much less clear.

\subsection{The Zwara-like degeneration defined by triangles}

Consider the degeneration $\leq_\Delta$ given by distinguished triangles. Then, it is not difficult to show that this degeneration concept is well-behaved with respect to
the image under a triangle functor.

\begin{lem}\label{Deltadegenunderimage}
Let $\mathcal C$ and $\mathcal D$ be triangulated categories and let
$$F:{\mathcal C}\longrightarrow{\mathcal D}$$
be a functor of triangulated categories. In particular $F$ sends
distinguished triangles to distinguished triangles. Then for all objects
$M$ and $N$ we get $$M\leq_{\Delta}N\Rightarrow F(M)\leq_{\Delta}F(N).$$
\end{lem}

\begin{proof} Indeed, suppose $M\leq_{\Delta}N$. Then there exists an object $Z$
such that $$Z\stackrel{v\choose u}{\rightarrow} Z\oplus M\rightarrow N\rightarrow Z[1]$$
is a distinguished triangle. We apply $F$ to this triangle, using that the hypothesis on $F$ implies that $F$ preserves finite direct sums, and using again the hypothesis on $F$, we obtain that
$$F(Z)\stackrel{F(v)\choose F(u)}{\rightarrow} F(Z)\oplus F(M)\rightarrow F(N)\rightarrow F(Z)[1]$$
is a distinguished triangle by hypothesis on $F$.
Since $v$ is assumed to be nilpotent, $F(v)$ is also nilpotent.
Therefore $F(M)\leq_{\Delta}F(N)$ as claimed.
\end{proof}

However, if ${\mathcal C}$ is a full triangulated subcategory of
$\mathcal D$,
then $M\leq_\Delta N$ in $\mathcal D$ does not necessarily imply that
$M\leq_\Delta N$ in $\mathcal C$. Indeed, the object $Z$, which is
needed for the construction does not need to lie in $\mathcal C$.

\subsection{The Yoshino-like degeneration defined by degeneration data}

Contrary to the situation for $\leq_\Delta$ we get for the
geometrically inspired degeneration that $M\leq_{cdeg}N$ does
not imply necessarily $F(M)\leq_{cdeg}F(N)$.

The degeneration $\leq_{cdeg}$ is well-behaved with respect to
a fully faithful embedding of triangulated categories.

\begin{lem}\label{cdegwithrespecttoembeddings}
Let $k$ be a commutative ring, let ${\mathcal C}_k^\circ$ be a triangulated $k$-category and let $M$ and $N$ be
objects of ${\mathcal C}_k^\circ$. Suppose that ${\mathcal D}_k^\circ$ is a
triangulated $k$-category and suppose that $F:{\mathcal D}_k^\circ\rightarrow {\mathcal C}_k^\circ$ is a full embedding of triangulated categories. Then $F(M)\leq_{cdeg}F(N)$ implies that $M\leq_{cdeg}N$.
\end{lem}

\begin{proof} By definition we have a degeneration data
${\mathcal C}_k\stackrel{\uparrow_k^V}{\rightarrow} {\mathcal C}_V$ restricting to
${\mathcal C}_k^\circ\stackrel{\uparrow_k^V}{\rightarrow} {\mathcal C}_V^\circ$ and ${\mathcal C}_V^\circ\stackrel{\Phi}{\rightarrow} {\mathcal C}_k$
with an element $t:\text{id}_{{\mathcal C}_V}\rightarrow \text{id}_{{\mathcal C}_V}$ in the centre of ${\mathcal C}_V^\circ$. Moreover, we get an object $Q$ of ${\mathcal C}_V^\circ$ such that $\Phi(\text{cone}(t_{Q}))\simeq N$ in ${\mathcal C}_k^{\circ}$ and $p(Q)\simeq p(M\uparrow_k^V)$. Here
${\mathcal C}_V^\circ\stackrel{p}{\rightarrow}{\mathcal C}_V^\circ[t^{-1}]$ is the canonical functor.

But now, we may replace $\uparrow_k^V:{\mathcal C}_k^V\rightarrow {\mathcal C}_V^\circ$
by the composition $\uparrow_k^V\circ F:{\mathcal D}_k^\circ\rightarrow {\mathcal C}_V^\circ$
and obtain this way a degeneration data for ${\mathcal D}^\circ_k$, maintaining all the other data. The object $Q$ still serves for degeneration in ${\mathcal D}_k^\circ$. \end{proof}

\bigskip

Interesting is the case when Theorem~\ref{main} fully applies, combined with the above lemmas.

\begin{prop}
Let $k$ be a field and let ${\mathcal C}_k^\circ$ be the category of compact objects in an algebraic compactly generated triangulated $k$-category. If ${\mathcal D}^\circ_k$ is a full triangulated subcategory of ${\mathcal C}_k^\circ$, then for all objects
$M$ and $N$ of ${\mathcal D}_k^\circ$ we get that $M\leq_{cdeg}N$ with respect to
${\mathcal D}_k^\circ$ if and only if $M\leq_{cdeg}N$ in ${\mathcal C}_k^\circ$.
\end{prop}

\begin{proof} Suppose $M\leq_{cdeg}N$ with respect to
${\mathcal D}_k^\circ$.
Let ${\mathcal D}_k^\circ\stackrel{F}{\rightarrow}{\mathcal C}_k^\circ$ be the embedding functor. Then $M\leq_{\Delta}N$ in ${\mathcal D}_k^\circ$ by Theorem~\ref{main}, item 1.
By Lemma~\ref{Deltadegenunderimage} we obtain $FM\leq_{\Delta}FN$ in ${\mathcal C}_k^\circ$. But by Theorem~\ref{main}, item 2 we get that $FM\leq_{cdeg}FN$ with respect to ${\mathcal C}_k^\circ$.

Suppose $FM\leq_{cdeg}FN$ with respect to
${\mathcal C}_k^\circ$. Then Lemma~\ref{cdegwithrespecttoembeddings} directly gives that
$M\leq_{cdeg}N$ with respect to
${\mathcal D}_k^\circ$
\end{proof}

\section{Partial order}

A very important property of $\leq_{deg}$ is that it is a partial order
on the set of isomorphism classes of finite dimensional $A$-modules.
Yoshino showed that also $\leq_{stdvr}$ has a partial order property.
The question if $\leq_{\Delta}$ is a partial order is not easy, and
finiteness conditions are necessary. This is work due to Jensen, Su and
Zimmermann~\cite{JSZdegen}. The antisymmetricity in particular uses that
if $\mathcal T$ is an $R$-linear triangulated category for a commutative
ring $R$ such that $Hom_{\mathcal T}(X,Y)$ is of finite length as
$R$-modules for all objects $X$ and $Y$, then $M\leq_\Delta N$
implies that
$\text{length}_R(Hom_{\mathcal T}(X,M))\leq \text{length}_R(Hom_{\mathcal T}(X,N))$ for all objects $X$, and likewise for $Hom_{\mathcal T}(-,X)$.
If in addition there is $n$ such that $Hom_{\mathcal T}(M,N[n])=0$, then
$\text{length}_R(Hom_{\mathcal T}(X,M))= \text{length}_R(Hom_{\mathcal T}(X,N))$
for all $X$ implies that $M\simeq N$. The proof of this result is an
adaption of Bongartz proof in \cite{Bongartz}.

\begin{thm}(Jensen, Su, Zimmermann \cite{JSZdegen})
Let $R$ be a commutative ring and let $\mathcal T$ be an $R$-linear
skeletally small triangulated category with split idempotents
satisfying for any two objects $X,Y$ of $\mathcal T$
\begin{itemize}
\item  we get $\text{length}_R(Hom_{\mathcal T}(X,Y))<\infty$
\item  there is $n_{X,Y}\in{\mathbb Z}\setminus\{0\}$ such that
$Hom_{\mathcal T}(X,Y[n_{X,Y}])=0$
\end{itemize}
Then $\leq_\Delta$ is a partial order relation on the set of
isomorphism classes of objects in $\mathcal T$.
\end{thm}

As a last remark I want to mention that in \cite{SaorinZim} we
generalised this result slightly. The price we have to pay there
is that we need to consider the transitive hull of
$\leq_\Delta$.



\ifx\undefined\bysame
\newcommand{\bysame}{\leavevmode\hbox to3em{\hrulefill}\,}
\fi


\end{document}